\documentclass[12pt,reqno]{amsart}

\usepackage{palatino,mathpazo}
\usepackage{enumerate}
\usepackage{amsmath}
\usepackage{amsthm, bbm}
\usepackage{amssymb,exscale}
\usepackage{hyperref}
\hypersetup{ colorlinks=true, linkcolor=blue, citecolor=blue,
  filecolor=blue, urlcolor=blue}
  \usepackage{graphicx}
\usepackage{caption}
\usepackage{subcaption}

\usepackage[comma, sort&compress]{natbib}
\theoremstyle{plain} %documentation says there are only three styles
    \newtheorem{theorem}{Theorem}
    \newtheorem{lemma}[theorem]{Lemma}
    \newtheorem{proposition}[theorem]{Proposition}
    
    \newtheorem{claim}[theorem]{Claim}

\theoremstyle{definition} % For roman text in the body

\DeclareMathOperator{\R}{\mathbb{R}}

\DeclareMathOperator{\Z}{\mathbb{Z}}

\usepackage{mathtools}

\DeclarePairedDelimiter\floor{\lfloor}{\rfloor}

\DeclareMathOperator{\De}{d}

\DeclareMathOperator{\one}{\mathbbm{1}} %%%% indicator function

\newcommand{\E}{\mathsf{E}}
\newcommand{\Ex}[1]{\mathsf{E}\left[#1 \right]}
\newcommand{\var}[1]{\mathsf{Var}\left[#1 \right]}

\newcommand{\prob}{\mathsf{P}}
\DeclareMathOperator{\e}{e}
\renewcommand{\O}[1]{\mathrm{O}\left(#1\right)} %%%% for O asymptotics
\renewcommand{\o}[1]{\mathrm{o}\left(#1\right)} %%%% for o asymptotics
\newcommand{\eps}{\epsilon}
\newcommand{\eq}[1]{\begin{equation#1}}
\newcommand{\eeq}[1]{\end{equation#1}}
\newcommand{\eqa}[1]{\begin{eqnarray#1}}
\newcommand{\eeqa}[1]{\end{eqnarray#1}}
\newcommand{\vr}{\varphi}

\begin{document}

\title[Extremes of the supercritical Gaussian Free Field]{Extremes of the supercritical Gaussian Free Field}
\author[A. Chiarini]{Alberto Chiarini}
\thanks{The first author's research was supported by RTG 1845.}
\address{Technische Universit\"at Berlin,
MA 766, Strasse des 17. Juni 136, 10623
Berlin, Germany}
\email{chiarini@math.tu-berlin.de}

\author[A. Cipriani]{Alessandra Cipriani}
\address{Weierstrass-Institut, Mohrenstrasse 39, 10117 Berlin, Germany}
\email{Alessandra.Cipriani@wias-berlin.de}

\author[R. S. Hazra]{Rajat Subhra Hazra}
\address{Theoretical Statistics and Mathematics Unit, Indian Statistical Institute, 203, B.T. Road, Kolkata, 700108, India}
\email{rajatmaths@gmail.com}
\begin{abstract}
We show that the rescaled maximum of the discrete Gaussian Free Field (DGFF) in dimension larger or equal to $3$ is in the maximal domain of attraction of the Gumbel distribution. The result holds both for the infinite-volume field as well as the field with zero boundary conditions. We show that these results follow from an interesting application of the Stein-Chen method from \citet{AGG}.
\end{abstract}
\maketitle
\section{Introduction}
In this article we consider the problem of determining the scaling limit of the maximum of the discrete Gaussian free field (DGFF) on $\Z^d$, $d\ge 3$.  Recently the maximum of the DGFF in the critical dimension $d=2$ was determined in \cite{BrDiZe}. In this case, due to the presence of the logarithmic behavior of covariances, the problem is connected to extremes of various other models, for example the Branching Brownian motion and the Branching random walk. In $d\ge 3$, the presence of covariances decaying polynomially changes the setting but the behavior of maxima is still hard to determine \cite[Section 9.6]{Chatterjee}. This dependence also becomes a hurdle in various properties of level set percolation of the DGFF which were exhibited in a series of interesting works (\cite{PFASS, ASS, DrePF}). The behavior of local extremes in the critical dimension has also been unfolded recently in the papers \cite{BisLou, BisLou2}.

We consider the lattice $\Z^d$, $d\ge 3$ and take the infinite-volume Gaussian free field $(\vr_\alpha)_{\alpha\in \Z^d}$ with law $\prob$ on $\R^{\Z^d}$. The covariance structure of the field is given by the Green's function $g$ of the standard random walk, namely $\Ex{ \vr_\alpha\vr_\beta}= g(\alpha-\beta)$, for $\alpha,\,\beta\in \Z^d$. For more details of the model we refer to Section~\ref{sec:DGFF}. It is well- known (see for instance \cite{Lawler}) that  for $\alpha\neq \beta$, $g(\alpha-\beta)$ behaves likes $\|\alpha-\beta\|^{2-d}$ and hence for $\|\alpha-\beta\|\to+\infty$, the covariance goes to zero. However this is not enough to conclude that the scaling is the same of an independent ensemble. To give an example where this is not the case, when $V_N$ is the box of volume $N$, $\sum_{\alpha\in V_N}\vr_\alpha$ is of order $N^{1/{2}+1/d}$, unlike the i.~i.~d. setting (see for example \citet[Section 3.4]{Funaki}).

The expected maxima over a box of volume $N$ behaves like $\sqrt{2g(0) \log N}$. An independent proof of this fact is provided in Proposition~\ref{prop:LLN} below; this confirms the idea that the extremes of the field resemble that of independent $\mathcal N(0, g(0))$ random variables. In this article we show that the similarity is even deeper, since the fluctuations of the maximum after recentering and scaling converge to a Gumbel distribution. Note that in $d=2$ the limit is also Gumbel, but with a random shift (see \citet[Theorem 2.5]{BrDiZe}, \citet{BisLou}).
The main results of this article is the following.
\begin{theorem}\label{thm:main}
Let $A$ be a subset of $\Z^d$ with $|A|=N$\footnote{$|A|$ denotes the cardinality of $A$.}. We define two sequences as follows:
\begin{equation}\label{eq:cs}
b_N=\sqrt{g(0)}\left[\sqrt{2 \log N}-\frac{\log \log N+\log(4\pi)}{2\sqrt{2 \log N}}\right] \quad \text{ and }\quad a_N=g(0)(b_N)^{-1}
\end{equation}
so that for all $z\in \R$
$$
\lim_{N\to +\infty}\prob\left(\frac{\max_{\alpha\in A}\varphi_\alpha-b_N}{a_N}<z\right)=\exp(-\e^{-z}).
$$
\end{theorem}
Note that scaling and centering are exactly the same as in the i.~i.~d.  $\mathcal N(0, g(0))$ case, see for example \cite{Hall1982}. As in $d=2$, the argument depends on a comparison lemma. We show that in fact the proof is an interesting application of a Stein-Chen approximation result. Not only does the result depend on the correlation decay, but also crucially on the Markov property of the Gaussian free field. We use Theorem 1 of the paper by \cite{AGG} which approximates an appropriate dependent Binomial process with a Poisson process, and gives some calculable error terms. In general showing that the error terms go to zero is a non-trivial task. In the DGFF case, thanks to estimates on the Green's function and the Markov property, the error terms are negligible. \\
The techinques used for the infinite-volume DGFF allows us to draw conclusions also for the field with boundary conditions. For $n>0$ let $N:=n^d$; we consider the discrete hypercube $V_N:=[0,\,n-1]^d\cap \Z^d$. We define therein a mean zero Gaussian field $(\psi_\alpha)_{\alpha\in\Z^d}$ whose covariance matrix $(g_N(\alpha,\,\beta))_{\alpha,\,\beta\in V_N}$ is the Green's function of the discrete Laplacian with Dirichlet boundary conditions outside $V_N$ (again for a more precise definition see Section~\ref{sec:DGFF}). The convergence result is the following:
\begin{theorem}\label{thm:main2}
Let $V_N$ be as above and $(\psi_\alpha)_{\alpha\in \Z^d}$ be a DGFF with zero boundary conditions outside $V_N$ with law $\widetilde \prob_{V_N}$.  Let the centering and scaling be as in~\eqref{eq:cs}. Then for all $z\in \R$
$$
\lim_{N\to +\infty}\widetilde\prob_{V_N}\left(\frac{\max_{\alpha\in V_N}\psi_\alpha-b_N}{a_N}<z\right)=\exp(-\e^{-z}).
$$

\end{theorem}
The core of the proof is an application of Slepian's Lemma and a re-run of the proof of Theorem~\ref{thm:main}.\\
%\section{Preliminaries}
The structure of the article is as follows. In Section~\ref{sec:DGFF} we recall the main facts on the DGFF that will be used in Section~\ref{sec:main} to prove the main theorem.
\section{Preliminaries on the DGFF}\label{sec:DGFF}
Let $d\ge 3$ and  denote with $\|\,\cdot\,\|$ the $\ell_\infty$-norm on the lattice. Let $\psi=(\psi_\alpha)_{\alpha\in \Z^d}$ be a discrete Gaussian Free Field with zero boundary conditions outside $\Lambda\subset \Z^{d\phantom{d}}$. On the space $\Omega:=\R^{\Z^d}$ endowed with its product topology, its law $\widetilde \prob_\Lambda$ can be explicitly written as
$$\widetilde \prob_\Lambda(\De \psi)=\frac1{Z_\Lambda}\exp\left(-\frac1{4d}\sum_{\alpha,\,\beta\in \Z^d:\,\|\alpha-\beta\|=1}\left(\psi_\alpha-\psi_\beta\right)^2\right)\prod_{\alpha\in \Lambda}\De \psi_\alpha \prod_{\alpha\in \Z^d\setminus\Lambda}\delta_0(\De \psi_\alpha).$$
In other words $\psi_\alpha=0$ $\widetilde \prob_\Lambda$-a.~s. if $\alpha\in \Z^d\setminus \Lambda$, and $(\psi_\alpha)_{\alpha\in \Lambda}$ is a multivariate Gaussian random variable with mean zero and covariance $(g_\Lambda(\alpha,\,\beta))_{\alpha,\,\beta\in \Z^d}$, where $g_\Lambda$ is the Green's function of the discrete Laplacian problem with Dirichlet boundary conditions outside $\Lambda$. For a thorough review on the model the reader can refer for example to \cite{ASS}. It is known \cite[Chapter 13]{Georgii} that the finite-volume measure $\psi$ admits an infinite-volume limit as $\Lambda \uparrow \Z^d$ in the weak topology of probability measures. This field will be denoted as $\vr=(\vr_\alpha)_{\alpha\in \Z^d}$. It is a centered Gaussian field with covariance matrix $g(\alpha,\,\beta)$ for $\alpha,\,\beta\in \Z^d$. With a slight abuse of notation, we write $g(\alpha-\beta)$ for $g(0,\,\alpha-\beta)$ and also $g_\Lambda(\alpha)=g_\Lambda(\alpha,\,\alpha)$. It will be convenient for us to view $g$ through its random walk representation: if $\mathbb P_\alpha$ denotes the law of a simple random walk $S$ started at $\alpha\in \Z^d$, then
$$
g(\alpha,\,\beta)=\mathbb E_\alpha\left[\sum_{n\ge 0}\one_{\left\{S_n=\beta\right\}}\right].
$$
In particular this gives $g(0)<+\infty$ for $d\ge 3$.

A key fact for the Gaussian Free Field is its spatial Markov property, which will be used in the paper. The proof of the following Lemma can be found in \citet[Lemma 1.2]{PFASS}.
\begin{lemma}[Markov property of the Gaussian Free Field]\label{fact:MP}
 Let $\emptyset\neq K\Subset \Z^d$\footnote{$A\Subset B$ means that $A$ is a finite subset of $B$.}, $U:=\Z^d\setminus K$ and define $(\widetilde \vr_\alpha)_{\alpha\in \Z^d}$ by
 $$
 \vr_\alpha=\widetilde \vr_\alpha+\mu_\alpha,\quad \alpha\in \Z^d
 $$
 where $\mu_\alpha$ is the $\sigma(\vr_\beta, \,\beta\in K)$-measurable map defined as
 \begin{equation}\label{eq:drift}
 \mu_\alpha=\sum_{\beta\in K}\mathbb P_\alpha\left(H_{K}<+\infty,\,S_{H_{K}}=\beta\right)\vr_\beta,\quad \alpha\in \Z^d.
 \end{equation}
Here $H_K:=\inf\left\{n\ge 0:\,S_n\in K \right\}.$
Then, under $\prob$, $(\widetilde \vr_\alpha)_{\alpha\in \Z^d}$ is independent of $\sigma(\vr_\beta, \,\beta\in K)$ and distributed as $(\psi_\alpha)_{\alpha\in \Z^d}$ under $\widetilde \prob_{U}$.
\end{lemma}
As an immediate consequence of the Lemma (see \citet[Remark 1.3]{PFASS})
 $$\prob\left( (\vr_\alpha)_{\alpha\in \Z^d}\in \cdot\,|\sigma(\vr_\beta, \,\beta\in K)\right)=\widetilde \prob_U\left((\psi_\alpha+\mu_\alpha)_{\alpha\in \Z^d}\in \cdot \right)\quad \prob-a.~s. $$
 where $\mu_\alpha$ is given in \eqref{eq:drift}, $\widetilde \prob_U$ does not act on $(\mu_\alpha)_{\alpha\in \Z^d}$ and $(\psi_\alpha)_{\alpha\in \Z^d}$ has the law $\widetilde \prob_U$.

\subsubsection{Law of large numbers of the recentered maximum}
Although this can be obtained directly by Theorem~\ref{thm:main}, we think it is interesting to insert an independent proof of the behavior of the maximum of the DGFF.
\begin{proposition}[LLN for the maximum]\label{prop:LLN} Let $V_N:=[0,\,n-1]^d\cap\Z^d$, $N:=n^d>0$. The following limit holds:
$$
\lim_{N\to+\infty}\frac{\E\left[\max_{\alpha\in V_N}\vr_\alpha \right]}{\sqrt{2\log N}}= g(0).
$$
\end{proposition}
% \begin{lemma}[Fluctuations of the DGFF]
% The fluctuations of the $d$-dimensional discrete Gaussian Free Field in $d\ge 3$ are of order $1$.
% \end{lemma}
% \begin{proof}
% Borell's inequality \cite[Theorem 10]{ZeiBRW} tells us that that
% $$
% \prob\left( \left|\max_{x\in V_N}\vr_x-\Ex{\max_{x\in V_N}\vr_x}\right|>x\right)\le \exp\left(-\frac{x^2}{2} \sigma\right)
% $$
% where $\sigma=\max_x \var{\vr_x}$ is a finite constant. Hence by letting $Z_N:=\max_{x\in V_N}\vr_x-\Ex{\max_{x\in V_N}\vr_x}$ we have
% $$
% \Ex{\left(\max_{x\in V_N}\vr_x-\Ex{\max_{x\in V_N}\vr_x}\right)^2}=\int_0^{+\infty} x \prob(Z_N>x) \De x\le \int_0^{+\infty} x \exp\left(-\frac{x^2}{2} \sigma\right)\De x<+\infty
% $$
% uniformly in $N$.
% \end{proof}
\begin{proof}
Observe first that $g(0)\ge 1$ \cite[Exercise 1.5.7]{Lawler}.
The upper bound follows from \citet[Prop. 1.~1.~3]{Tal03} with $\tau:= g(0)$ and $M:=N$.
As for the lower bound, we will use Sudakov-Fernique inequality \cite[Theorem 2.~2.~3]{AdlerTaylor}. We first need a lower bound for $d(\alpha,\,\beta):=\sqrt{\Ex{\left(\vr_\alpha-\vr_\beta\right)^2}}$: we will apply here the bound
\begin{eqnarray}
%&&g(0)=1+\frac1{2d}+\o{d^{-1}},\quad d\to+\infty,\label{eq:one}\\
&&g(\alpha)\le\left(\frac{c\sqrt d}{\Vert \alpha\Vert} \right)^{d-2},\quad \Vert \alpha\Vert \ge d\label{eq:two}
\end{eqnarray}
whose proof is provided in \cite{ASSLow}. The key to obtain the result is to use a diluted version of the DGFF as follows. Consider $V_N^{(k)}:=V_N\cap k\Z^d$, where $k:=\floor*{\log n}\in \{1,\,2,\,\ldots\}$.
%so that
%\eq{}\label{eq:ineq}
%g(0)-\left(\frac{c\sqrt d}{\floor*{\log n}}\right)^{d-2}>0\iff n\ge \ceil*{\exp\left(\frac{c\sqrt d}{g(0)^{\frac1{d-2}}}\right)}
%\eeq{}
Note the fact that the expected maximum on $V_N$ is lower bounded by that on the diluted lattice $V_N^{(k)}$.
%\eq{}\label{eq:maxim}
%\Ex{\max_{\alpha\in V_N}\vr_\alpha}\ge \Ex{\max_{\alpha\in V_N^{(k)}}\vr_\alpha}.
%\eeq{}
Now for $\alpha,\,\beta\in T:=V_N^{(k)}$ and $k>d$
\eqa{*}
&&d(\alpha,\,\beta)=\sqrt{2 g(0)-2g(\alpha-\beta)}\stackrel{\eqref{eq:two}}{\ge}\sqrt 2 \sqrt{g(0)-\left(\frac{c\sqrt d}{\Vert \alpha-\beta\Vert} \right)^{d-2}}\\
&&\ge \sqrt 2 \sqrt{g(0)-\left(\frac{c\sqrt d}{\floor*{\log n}}\right)^{d-2}}=:\nu(n,\,d){>}0
\eeqa{*}
for $n$ large enough.
Notice also that $\lim_{N\to+\infty}\nu(n,\,d)=\sqrt{2 g(0)}.$
Hence by an application of the Sudakov-Fernique inequality
\eqa{*}
&&\frac{\Ex{\max_{\alpha\in V_N}\vr_\alpha}}{\sqrt{ \log N}}\ge \nu(n,\,d)\sqrt{\frac{\log \left| T\right|}{\log N}}.\eeqa{*}
%By setting
%$$
%k:=1-\frac1d\frac{\log \log N}{\log N}
%$$
We obtain $\log \left| T\right|=d\log \floor*{\frac{n}{k}}(1+\o{1})=d\log \floor*{\frac{n}{\floor*{\log n}}}(1+\o{1})$\footnote{$f(N)=\o{1}$ means $\lim_{N\to+\infty}f(N)=0.$}. It follows that $\frac{\log \left| T\right|}{\log N}=1+\o{1}$ and
$$
\lim_{N\to+\infty}\frac{\Ex{\max_{\alpha \in V_N}\vr_\alpha}}{\sqrt{\log N}}\ge {}\sqrt{2 g(0)}.
$$
\end{proof}
\section{Proof of the main result}\label{sec:main}
The proof of the main result is an application of the Stein-Chen method. To keep the article self contained we recall the result from \cite{AGG}.
\subsection{Poisson approximation for extremes of random variables}\label{subsec:Po}
The main tool we will use relies on a two-moment condition to determine the convergence of the number of exceedances for a sequence of random variables. Let $(X_\alpha)_{\alpha\in  A}$ be a sequence of (possibly dependent) Bernoulli random variables of parameter $p_\alpha$. Let $W:=\sum_{\alpha\in A}X_\alpha$ and $\lambda:=\Ex{W}$. Now for each $\alpha$ we assume the existence of a subset $B_\alpha\subseteq  A$ which we consider a ``neighborhood'' of dependence for the variable $X_\alpha$, such that $X_\alpha$ is nearly independent from $X_\beta$ if $\beta\in A\setminus B_\alpha$. Set
$$
b_1:=\sum_{\alpha\in A}\sum_{\beta\in B_\alpha}p_\alpha p_\beta,
$$
$$
b_2:=\sum_{\alpha\in A}\sum_{\alpha\neq \beta\in B_\alpha}\Ex{X_\alpha X_\beta},
$$
$$
b_3:=\sum_{\alpha\in A}\Ex{\left|\Ex{X_\alpha-p_\alpha\left|\right.\mathcal H_\alpha}\right|}
$$
where
$$
\mathcal H_\alpha:=\sigma\left(X_\beta:\,\beta\in A\setminus B_\alpha\right).
$$
\begin{theorem}[Theorem 1, \cite{AGG}]\label{thm:AGG}
Let $Z$ be a Poisson random variable with $\Ex{Z}=\lambda$ and let $d_{TV}$ denote the total variation distance between probability measures. Then
$$
d_{TV}(\mathcal L(W),\,\mathcal L(Z))\le 2(b_1+b_2+b_3)
$$
and
$$
\left|P(W=0)-\e^{-\lambda} \right|<\min\left\{1,\,\lambda^{-1}\right\}(b_1+b_2+b_3).
$$
\end{theorem}

Let now $A\Subset \Z^d$ with $N:=|A|$, $u_N(z):=a_N z + b_N$, and define for all $\alpha\in A$
$$
X_\alpha=\one_{\left\{\varphi_\alpha>u_N(z)\right\}}\sim Be(p_\alpha).
$$
A standard tool to determine the asymptotic of $p$ is Mills ratio:
\eqa{}\label{eq:Mills}
\left( 1-\frac1{t^2}\right)\frac{\e^{-{t^2/2}}}{\sqrt{2\pi}t}\le \prob\left(\mathcal N(0,\,1)>t \right)\le \frac{\e^{-{t^2/2}}}{\sqrt{2\pi}t},\quad t>0.
\eeqa{}
This yields $p_\alpha\sim N^{-1}\exp(-z)$\footnote{$f\sim g$ means that $\lim_{N\to+\infty} f(N)/g(N)=1.$}. Since $p_\alpha$ is independent of $\alpha$, we suppress the subscript $\alpha$ throughout. We furthermore introduce
$
W:=\sum_{\alpha\in A} X_\alpha
$
and see that $\Ex{W}\sim \e^{-z}$. Of course $W$ is closely related to the maximum since $\left\{\max_{\alpha\in A}\vr_\alpha\le u_N(z)\right\}=\left\{W=0 \right\}$. We will now fix $z\in \R$ and $\lambda:=\e^{-z}$. We are now ready to prove our main result.
\begin{proof}
Our main idea is to apply Theorem~\ref{thm:AGG}. The proof will first show that the limit is Gumbel, and in the second part we will prove uniform convergence. To this scope we define, for a fixed but small $\eps>0$,
$$B_\alpha:=B\left(\alpha,\,(\log N)^{2+2\eps}\right)\cap A $$
where $B(\alpha,\,L)$ denotes the ball of center $\alpha$ of radius $L$ in the $\ell_\infty$-distance. We draw below examples of such neighborhoods when $\alpha \in \partial A:=\left\{\gamma\in A:\,\exists\,\beta\in \Z^d\setminus A,\,\|\beta-\gamma\|=1 \right\}$ and $\alpha\in\mathrm{int}(A)=A\setminus\partial A$.
\begin{figure}[!ht]
        \centering
        \begin{subfigure}[b]{0.5\textwidth}
                \includegraphics[width=\textwidth]{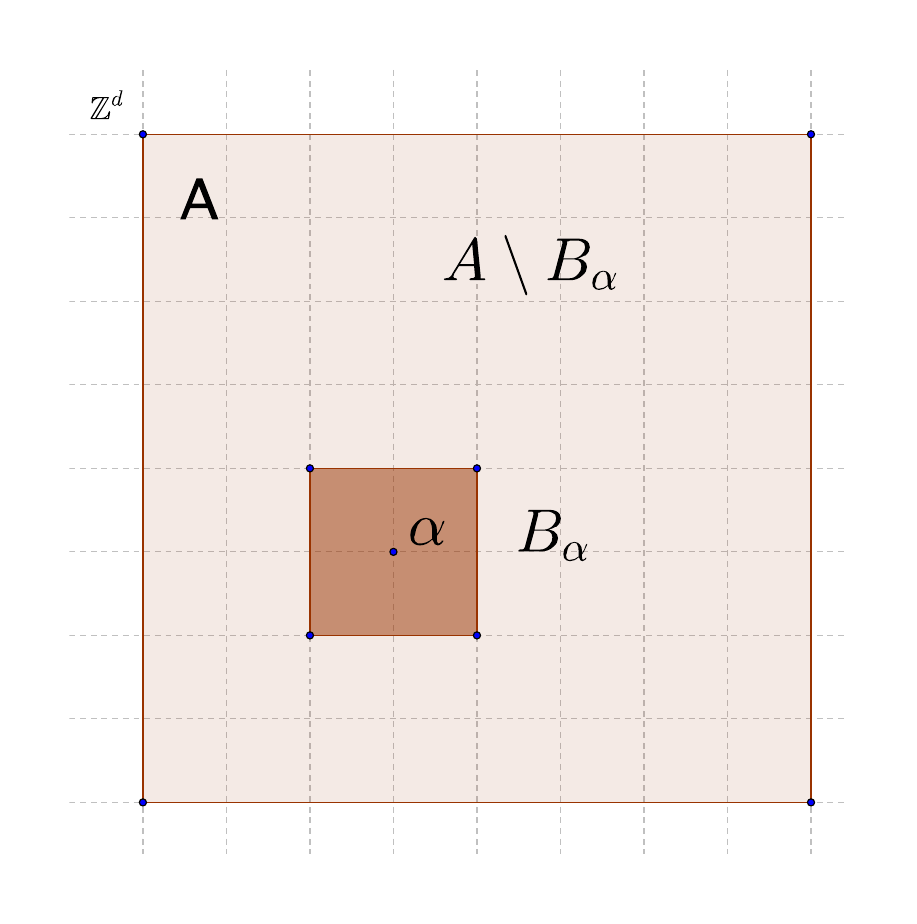}
                \caption{$B_\alpha$ when $\alpha\in \mathrm{int}(A)$.}
                \label{fig:NonBoundary}
        \end{subfigure}%
        \begin{subfigure}[b]{0.5\textwidth}
                \includegraphics[width=\textwidth]{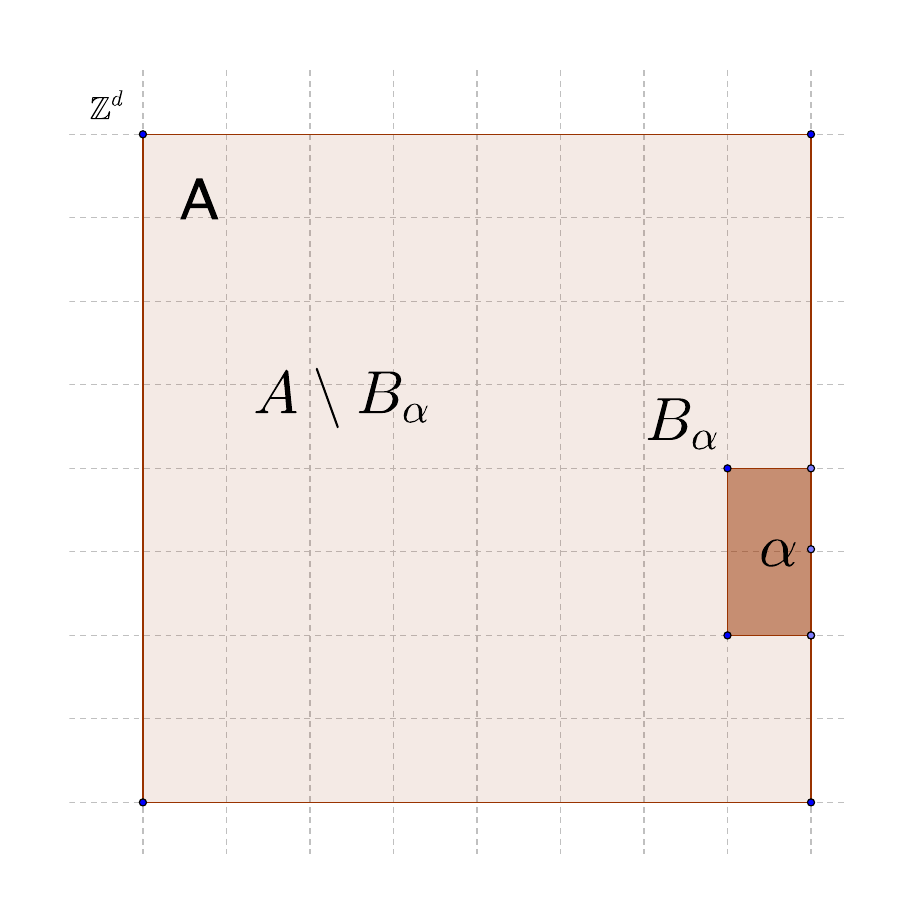}
                \caption{$B_\alpha$ when $\alpha\in \partial A.$}
                \label{fig:Boundary}
        \end{subfigure}
           \caption{Examples of $B_\alpha$}\label{fig:Balpha}
\end{figure}
\\

{\bf Convergence.}
The method is based on the estimate of three terms (cf. Subsec.~\ref{subsec:Po}).\\
\indent (i) Recall $b_1=\sum_{\alpha\in A}\sum_{\beta\in B_\alpha}p^2$. Using Mills ratio we have
\eqa{}
b_1&\le& c N(\log N)^{d(2+2\eps)}\left(\frac{\sqrt{g(0)}\e^{-\frac1{2g(0)} u_N(z)^2}}{\sqrt{2\pi }u_N(z)}\right)^2\nonumber\\
&=& N^{-1}(\log N)^{d(2+2\eps)}\e^{-2z+\o{1}}=\o{1}.\label{eq:b_1}
\eeqa{}
\indent (ii) Recall $b_2=\sum_{\alpha\in A}\sum_{\alpha\neq \beta\in B_\alpha}\Ex{X_\alpha X_\beta}$. First we need to estimate the joint probability
$$\prob\left(\vr_\alpha>u_N(z), \vr_\beta>u_N(z)\right).$$
Denote the covariance matrix
$$\Sigma_2=\begin{bmatrix}
g(0) & g(\alpha-\beta)\\
g(\alpha-\beta) & g(0) \\
\end{bmatrix}$$
Note that, for $w\in \R^2$, one has
$$w^t\Sigma_2^{-1} w=\frac1{g(0)^2-g(\alpha-\beta)^2} \left( g(0)\left(w_1^2+w_2^2\right)-2g(\alpha-\beta) w_1w_2\right).$$
Using $\mathbf{1}:=(1,1)^t$ we denote by
$$\Delta_i:=u_N(z) \left(\mathbf{1}^t \Sigma_2^{-1}\right)_i= \frac{u_N(z) (g(0)-g(\alpha-\beta))}{g(0)^2-g(\alpha-\beta)^2}=\frac{u_N(z)}{g(0)+g(\alpha-\beta)},\quad i=1,\,2.$$
Exploiting an easy upper bound on bi-variate Gaussian tails (see \cite{Savage}) we have
\begin{equation}\label{eq:b_2:1}
\prob(\vr_\alpha>u_N(z), \vr_\beta>u_N(z))\le \frac1{2\pi} \frac1{|\det \Sigma_2|^{1/2}\Delta_1\Delta_2}\exp\left(-\frac{u_N(z)^2}{2}  \mathbf{1}^t \Sigma_2^{-1}\mathbf 1\right)
\end{equation}
Note that using the explicit formula for  the determinant one can bound the first factor easily by
$$\frac1{2\pi} \frac1{|\det \Sigma_2|^{1/2}\Delta_1\Delta_2}\le  \frac{\left(1+ \frac{g(\alpha-\beta)}{g(0)}\right)^{3/2}}{\left(1-\frac{g(\alpha-\beta)}{g(0)}\right)^{1/2}}.$$
Now using $u_N(z)^2= b_N^2+ 2g(0)z+g(0)^2 z^2/b_N^2$ and the bound of $b_N^2$ \cite[Equation 3]{Hall1982}
$$g(0)(2\log N-\log \log N-\log 4\pi)\le b_N^2\le 2g(0)\log N$$
we can now upper bound the exponential term by,
$$\exp\left(-\frac{u_N(z)^2}{2}  \mathbf{1}^t \Sigma_2^{-1}\mathbf 1\right)\le  N^{-\frac{2g(0)}{g(0)+g(\alpha-\beta)}}\e^{-\frac{2g(0)z}{g(0)+g(\alpha-\beta)}+\o{1}}.$$

%&= \frac1{2\pi} \frac{(g(0)+g(\alpha-\beta))^2}{(g(0)^2-g(\alpha-\beta)^2)^{1/2}u_N(z)^2}\exp\left(-\frac{u_N(z)^2}{2}  \frac{2(g(0)-g(\alpha-\beta))}{g(0)^2-g(\alpha-\beta)^2}\right)  \\
%&\le \frac1{4\pi \log N} \frac{\left(1+ \frac{g(\alpha-\beta)}{g(0)}\right)^{3/2}}{\left(1-\frac{g(\alpha-\beta)}{g(0)}\right)^{1/2}} N^{-\frac{2g(0)}{g(0)+g(\alpha-\beta)}}\left(4\pi \log N \right)^{\frac{g(0)}{g(0)+g(\alpha-\beta)}}\e^{-\frac{2g(0)z}{g(0)+g(\alpha-\beta)}+\o{1}}\\
%&\le \frac{\left(1+ \frac{g(\alpha-\beta)}{g(0)}\right)^{3/2}}{\left(1-\frac{g(\alpha-\beta)}{g(0)}\right)^{1/2}} N^{-\frac{2g(0)}{g(0)+g(\alpha-\beta)}}\e^{-\frac{2g(0)z}{g(0)+g(\alpha-\beta)}+\o{1}}
%\end{align*}

Also note that for $x\neq 0$, $g(\|x\|)/g(0)\le g(e_1)/g(0)=1-\kappa$ where $\kappa:= \mathbb P_0\left( \widetilde H_0=+\infty\right)\in (0,1)$ and $\widetilde H_0=\inf\left\{n\ge 1:\,S_n=0 \right\}$. Hence we have that
$$\frac{g(0)}{g(0)+g(\alpha-\beta)}\ge \frac1{2-\kappa} \qquad\text{ and } \qquad\frac{g(\alpha-\beta)}{g(0)+g(\alpha-\beta)}\le 1-\kappa.$$
We obtain thus
\eqa{*}&&\prob\left( \vr_\alpha> u_N(z), \vr_\beta>u_N(z)\right)\le \frac{(2-\kappa)^{3/2}}{\kappa^{1/2}} N^{-\frac{2}{(2-\kappa)}} \max\left( \e^{-2z}\one_{\left\{z\le 0\right\}}, \e^{-2z/(2-\kappa)}\one_{\left\{z>0\right\}}\right).
\eeqa{*}
We get finally for some constants $c,\,c'>0$ depending only on $d$ and $\kappa$
\begin{align}
b_2&\le c N (\log N)^{d(2+2\eps)} \frac{(2-\kappa)^{3/2}}{\kappa^{1/2}} N^{-\frac{2}{(2-\kappa)}}\max\left( \e^{-2z}\one_{\left\{z\le 0\right\}}, \e^{-2z/(2-\kappa)}\one_{\left\{z>0\right\}}\right)\nonumber\\
&\le c' N^{-\frac{\kappa}{(2-\kappa)}} ( \log N)^{d(2+2\eps)} \max\left( \e^{-2z}\one_{\left\{z\le 0\right\}}, \e^{-2z/(2-\kappa)}\one_{\left\{z>0\right\}}\right).\label{eq:b_2}
\end{align}
Since $\kappa/(2-\kappa)>0$ we have that $b_2=\o{1}$. \\

\indent (iii) Recall $b_3=\sum_{\alpha\in A}\Ex{\left|\Ex{X_\alpha-p_\alpha\left|\right.\mathcal H_\alpha}\right|}$.
It will be convenient to introduce also another $\sigma$-algebra which strictly contains $\mathcal H_\alpha=\sigma\left(X_\beta:\,\beta \in A\setminus B_\alpha\right)$, that is
$$
\mathcal H_\alpha'=\sigma\left(\vr_\beta:\,\beta\in A\setminus B_\alpha\right).
$$
Using the tower property of the conditional expectation and Jensen's inequality
\begin{align*}
\Ex{\left|\Ex{X_\alpha-p\left|\right.\mathcal H_\alpha}\right|}\le\Ex{\left|\Ex{X_\alpha-p\left|\right.\mathcal H_\alpha'}\right|}.
\end{align*}
At this point we recognize, thanks to Corollary~\ref{fact:MP}, that
$$
\Ex{X_\alpha\left|\right.\mathcal H_\alpha'}=\widetilde \prob_{\Z^d\setminus(A\setminus B_\alpha)}(\psi_\alpha+\mu_\alpha>u_N(z))\quad \prob-a.~s.
$$
where $(\psi_\alpha)_{\alpha\in \Z^d}$ is a Gaussian Free Field with zero boundary conditions outside $A\setminus B_\alpha$. In addition, observe that $g_{U_\alpha}(\alpha)\le g(0)$ \cite[Section 1.5]{Lawler}. We will write more compactly $U_\alpha:=\Z^d\setminus(A\setminus B_\alpha)$.

We will make use of the fact that $\mu_\alpha$ is a centered Gaussian, and apply the same estimates of \cite{BP}: first observe using strong Markov property we have $\beta\in A\setminus B_\alpha$,
\begin{equation}\label{eq:Markov}
g(\alpha,\beta)=\sum_{\gamma\in A\setminus B_\alpha}\mathbb P_\alpha\left(H_{A\setminus B_\alpha}<+\infty,\,S_{H_{A\setminus B_\alpha}}=\gamma\right)g(\gamma,\beta).
\end{equation}
%Denoting by $S_\cdot\circ\theta_m=S_{m+\cdot}$ the time shift by $m$ of the random walk, we observe that for $\beta\in A\setminus B_\alpha$
%\eqa{}
%g(\alpha,\beta)&=&\mathbb E_\alpha\left[\sum_{n\ge 0}\one_{\left\{S_n=\beta\right\}}\right]=\mathbb E_\alpha\left[\left(\sum_{n\ge 0}\one_{\left\{S_n=\beta\right\}}\right)\circ \theta_{H_{A\setminus B_\alpha}}\right]\nonumber\\
%&=&\mathbb E_\alpha\left[\mathbb E_{S_{H_{A\setminus B_\alpha}}}\left[\sum_{n\ge 0}\one_{\left\{S_n=\beta\right\}}\right] \right]=\mathbb E_\alpha\left[g\left(S_{H_{A\setminus B_\alpha}},\beta\right),\,H_{A\setminus B_\alpha}<+\infty\right]\nonumber\\
%&=&\sum_{\gamma\in A\setminus B_\alpha}\mathbb P_\alpha\left(H_{A\setminus B_\alpha}<+\infty,\,S_{H_{A\setminus B_\alpha}}=\gamma\right)g(\gamma,\beta)\label{eq:Markov}.
%\eeqa{}
We can plug this in to obtain
\begin{align}
&\var{\mu_\alpha}
%&=\sum_{\beta,\,\gamma\in A\setminus B_\alpha}\mathbb P_\alpha\left(H_{A\setminus B_\alpha}<+\infty,\,S_{H_{A\setminus B_\alpha}}=\beta\right)\mathbb P_\alpha\left(H_{A\setminus B_\alpha}<+\infty,\,S_{H_{A\setminus B_\alpha}}=\gamma\right) g(\beta,\,\gamma)\nonumber\\
\stackrel{\eqref{eq:Markov}}{=}\sum_{\beta\in A\setminus B_\alpha}\mathbb P_\alpha\left(H_{A\setminus B_\alpha}<+\infty,\,S_{H_{A\setminus B_\alpha}}=\beta\right)g(\alpha,\,\beta)\le \sup_{\beta\in A\setminus B_\alpha}g(\alpha,\,\beta)\nonumber\\
&\le \frac{c}{(\log N)^{2(1+\eps)(d-2)}}\label{eq:barza}
\end{align}
by the standard estimates for the Green's function
\eqa{}\label{eq:bound_g}c_d\|\alpha-\beta\|^{2-d}\le g(\alpha,\,\beta)\le C_d\|\alpha-\beta\|^{2-d}\eeqa{}
for some $0<c_d\le C_d<+\infty$ independent of $\alpha$ and $\beta$ \cite[Theorem 1.~5.~4]{Lawler}.
Using the estimate
\eq{}\label{eq:prova}
\prob\left(\left|\mathcal N(0,1)\right|>a\right)\le \e^{-{a^2/2}},\;a>0\eeq{}
\vspace{-0.2cm}
we get that there exists a constant $C>0$ such that
\eq{}\label{eq:rate_zero}
\prob\left(|\mu_\alpha|>\left(u_N(z)\right)^{-1-\eps} \right)\le C \exp\left(-(\log N)^{(2d-5)(1+\eps)} \right).\eeq{}
Note that this quantity goes to zero since $d\ge 3$. Hence
\eqa{*}
&&\Ex{\left|\widetilde \prob_{U_\alpha}(\psi_\alpha+\mu_\alpha>u_N(z))-p\right|}=\Ex{\left|\widetilde \prob_{U_\alpha}(\psi_\alpha+\mu_\alpha>u_N(z))-p\right|\one_{\left\{|\mu_\alpha|\le\left(u_N(z)\right)^{-1-\eps} \right\}}}\\
&&+\Ex{\left|\widetilde \prob_{U_\alpha}(\psi_\alpha+\mu_\alpha>u_N(z))-p\right|\one_{\left\{|\mu_\alpha|>\left(u_N(z)\right)^{-1-\eps} \right\}}}=:T_1+T_2.\eeqa{*}
By \eqref{eq:rate_zero} and the fact that $d\ge 3$, we notice that $N T_2=\o{1}$. Therefore it is sufficient to treat the term $T_1$. By conditioning on whether $p$ is larger or smaller than $\prob_{U_\alpha}(\psi_\alpha+\mu_\alpha> u_N(z))$ we can split the event in $T_1$ into the following two terms.
\eqa{}
%&&\Ex{\left|\widetilde \prob_{U_\alpha}(\psi_\alpha+\mu_\alpha>u_N(z))-p\right|\one_{\left\{|\mu_\alpha|\le\left(u_N(z)\right)^{-1-\eps} \right\}}}\nonumber\\
&&\Ex{\left(\widetilde \prob_{U_\alpha}(\psi_\alpha+\mu_\alpha>u_N(z))-p\right)\one_{\left\{|\mu_\alpha|\le\left(u_N(z)\right)^{-1-\eps} \right\}}\one_{\left\{p<\widetilde \prob_{U_\alpha}(\psi_\alpha+\mu_\alpha>u_N(z))\right\}}}\nonumber\\
&&+\Ex{\left(p-\widetilde \prob_{U_\alpha}(\psi_\alpha+\mu_\alpha>u_N(z))\right)\one_{\left\{|\mu_\alpha|\le\left(u_N(z)\right)^{-1-\eps} \right\}}\one_{\left\{p\ge \widetilde \prob_{U_\alpha}(\psi_\alpha+\mu_\alpha> u_N(z))\right\}}}\nonumber\\
&&=:T_{1,1}+T_{1,2}.\label{eq:treat}
\eeqa{}
We will now deal with $T_{1,2}$. The first one can be treated with a similar calculation.  Using Mill's ratio and fact that $\psi_\alpha$ has variance $g_{U_\alpha}$ we get that,
{\small\eqa{}
&&p-\widetilde \prob_{U_\alpha}(\psi_\alpha+\mu_\alpha>u_N(z))\nonumber\\
&&\le \frac{\sqrt{g(0)}\e^{-\frac{u_N(z)^2}{2g(0)}}}{\sqrt{2\pi}u_N(z)}-\left(1-\left(\frac{\sqrt{g_{U_\alpha}(\alpha)}}{u_N(z)-\mu_\alpha}\right)^{2}\right)\frac{\sqrt{g_{U_\alpha}(\alpha)}\e^{-\frac{(u_N(z)-\mu_\alpha)^2}{2 g_{U_\alpha}(\alpha)}}}{\sqrt{2\pi }(u_N(z)-\mu_\alpha)}\nonumber\\
\eeqa{}}
We have on the event  $\left\{|\mu_\alpha|\le\left(u_N(z)\right)^{-1-\eps} \right\}$ that the above is bounded by
\begin{equation}\label{eq:here2}
\frac{\sqrt{g(0)}\e^{-\frac{u_N(z)^2}{2g(0)}}}{\sqrt{2\pi }u_N(z)}\left(1-\left(1+\o{1}\right)\frac{\sqrt{g_{U_\alpha}(\alpha)}u_N(z)\e^{\left(1-\frac{g(0)}{g_{U_\alpha}(\alpha)}\right)\frac{u_N(z)^2}{2g(0)}+\frac{u_N(z)^{-\eps}}{g_{U_\alpha}(\alpha)}-\frac{u_N(z)^{-2-2\eps}}{2g_{U_\alpha}(\alpha)}}}{\sqrt{ g(0)}u_N(z)(1-u_N(z)^{-2-\eps})} \right).
\end{equation}
%&&\le  \frac{\sqrt{g(0)}\e^{-\frac{u_N(z)^2}{2g(0)}}}{\sqrt{2\pi }u_N(z)}\left(1-\left(1+\o{1}\right)\frac{\sqrt{g_{U_\alpha}(\alpha)}u_N(z)\e^{\left(1-\frac{g(0)}{g_{U_\alpha}(\alpha)}\right)\frac{u_N(z)^2}{2g(0)}+\frac{\mu_\alpha u_N(z)}{g_{U_\alpha}(\alpha)}-\frac{\mu_\alpha^2}{2g_{U_\alpha}(\alpha)}}}{\sqrt{ g(0)}(u_N(z)-\mu_\alpha)} \right)\nonumber\\
%&&=\frac{\sqrt{g(0)}\e^{-\frac{u_N(z)^2}{2g(0)}}}{\sqrt{2\pi }u_N(z)}\left(1-\left(1+\o{1}\right)\frac{\sqrt{g_{U_\alpha}(\alpha)}u_N(z)\e^{\left(1-\frac{g(0)}{g_{U_\alpha}(\alpha)}\right)\frac{u_N(z)^2}{2g(0)}+\frac{u_N(z)^{-\eps}}{g_{U_\alpha}(\alpha)}-\frac{u_N(z)^{-2-2\eps}}{2g_{U_\alpha}(\alpha)}}}{\sqrt{ g(0)}u_N(z)(1-u_N(z)^{-2-\eps})} \right).\label{eq:here2}
%\eeqa{}}
Since the bound is non random, by bounding the indicator functions by $1$,
$$
\Ex{\left(p-\widetilde \prob_{U_\alpha}(\psi_\alpha+\mu_\alpha>u_N(z))\right)\one_{\left\{|\mu_\alpha|\le\left(u_N(z)\right)^{-1-\eps} \right\}}\one_{\left\{p\ge\widetilde \prob_{U_\alpha}(\psi_\alpha+\mu_\alpha\le u_N(z))\right\}}}\le\eqref{eq:here2}.
$$
Now
\eq{}\label{eq:tap}
b_3\le \sum_{\alpha\in A}(T_1+T_2)\stackrel{\eqref{eq:rate_zero}}\le \sum_{\alpha\in A}T_1 +\o{1}=\sum_{\alpha\in A}T_{1,1}+\sum_{\alpha\in A}T_{1,2}+\o{1}.
\eeq{}
Then
\eqa{*}
T_{1,2}=\frac{\sqrt{g(0)}\e^{-\frac{u_N(z)^2}{2g(0)}}}{\sqrt{2\pi }u_N(z)} \left(1-(1+\o{1})\left(\frac{\sqrt{g_{U_\alpha}(\alpha)}u_N(z)\e^{\left(1-\frac{g(0)}{g_{U_\alpha}(\alpha)}\right)\frac{u_N(z)^2}{2g(0)}+\o{1}}}{\sqrt{g(0) }u_N(z)(1+\o{1})}\right)\right).
\eeqa{*}
Observe that $1-\frac{g(0)}{g_{U_\alpha}(\alpha)}<0$ since $g(0)>g_{U_\alpha}(\alpha).$ We observe further that (and we will prove it in a moment)
\begin{claim}\label{claim}
$\sup_{\alpha\in A} \left(1-\frac{g(0)}{g_{U_\alpha}(\alpha)}\right)u_N(z)^2=\o{1}.$
\end{claim}
Therefore $T_{1,2}=\o{1}$ uniformly in $\alpha$. This yields that
\eqa{}\label{eq:tapp}\sum_{\alpha\in\mathcal A}T_{1,2}\le N\frac{\sqrt{g(0)}\e^{-\frac{u_N(z)^2}{2g(0)}}}{\sqrt{2\pi }u_N(z)}\o{1}=\e^{-z+\o{1}}\o{1}.\label{eq:T_12} \eeqa{}
Analogously, $ \sum_{\alpha\in\mathcal A}T_{1,\,1}=\o{1}$.
Plugging \eqref{eq:tapp} in \eqref{eq:tap}, one obtains $b_3=\o{1}$.
%In particular, a standard computation yields $b_3\le c (\log N)^{-1+2(d-2)(1+\eps)}$ for some $c>0$.\\
We now only need to show Claim~\ref{claim}. By the Markov property we know
$$g_{U_\alpha}(\alpha)= g(0)-\sum_{\gamma\in A\setminus B_\alpha} \mathbb P_\alpha \left( H_{A\setminus B_\alpha}<+\infty, \,S_{H_{A\setminus B_\alpha}}=\gamma\right)g(\gamma, \alpha).$$
This shows that
$$0\le \frac{g(0)}{g_{U_\alpha}(\alpha)}-1\le \frac{\sup_{\gamma\in A\setminus B_\alpha}g(\gamma,\alpha)}{g_{U_\alpha}(\alpha)}.$$
Note that $g(\gamma, \alpha)\stackrel{\eqref{eq:bound_g}}{\le} C_d(\log N)^{-2(d-2)(1+\epsilon)}$. Also,
$g_{U_\alpha}(\alpha)= \mathbb E_\alpha\left[ \sum_{n=0}^{H_{A\setminus B_\alpha}} \one_{\left\{S_n=\alpha\right\}}\right]\ge 1$ and hence we have
\eq{}\label{eq:bizzo}0\leq\frac{g(0)}{g_{U_\alpha}(\alpha)}-1\le c(\log N)^{-2(d-2)(1+\epsilon)}\eeq{}
from which it follows that
\eq{}\label{eq:abba}\left(1-\frac{g(0)}{g_{U_\alpha}(\alpha)}\right)u_N(z)^2 \le c (\log N)^{-2(d-2)(1+\epsilon)}(\log N+z+\o{1})=\o{1}.\eeq{}
Therefore the claim follows and we have shown pointwise convergence.

\end{proof}
\subsection{DGFF with boundary conditions: proof of Theorem~\ref{thm:main2}}
The idea of the proof is to exploit the convergence we have obtained in the previous section. We will show a lower bound through a comparison with i.~i.~d. variables, and an upper bound by considering the maximum restricted to the bulk of $V_N$, concluding by means of a convergence-of-types result. We abbreviate $g_N(\cdot,\,\cdot):=g_{V_N}(\cdot,\,\cdot)$.
For $\delta>0$ define (recall that $V_N=[0,\,n-1]^d\cap\Z^d$, with $N=n^d$)
$$V_N^\delta:=\left\{\alpha\in V_N:\,\|\alpha-\gamma\|>\delta N^{1/d},\,\gamma\in \Z^d\setminus V_N\, \right\}.$$
%We remark that it is sufficient to show pointwise convergence,
%as uniform convergence would follow again from \citet[Satz I]{Polya}.

We begin with the easier lower bound.
\begin{proof}[Proof of Theorem~\ref{thm:main2}: lower bound]
We will need a lower and an upper bound on the limiting distribution of the maximum. Let us start with the former. We use the shortcut $\widetilde\prob_N:=\widetilde\prob_{V_N}$. First we note that since the covariance of $(\psi_\alpha)$ is non-negative, we can apply Slepian's lemma for the lower bound. Let $(Z_\alpha)_{\alpha\in V_N}$ be independent mean zero Gaussian variables with variance $g_N(\alpha)$; then by Slepian's lemma it follows that
$$\widetilde\prob_N\left(\max_{\alpha\in V_N} Z_\alpha\le u_N(z)\right)\le \widetilde \prob_N\left(\max_{\alpha\in V_N} \psi_\alpha\le u_N(z)\right),$$
where $u_N(z)=a_Nz+b_N$ as before. Then we want to analyze $\prob (\max_{\alpha\in A} Z_\alpha\le u_N(z))$. First fix $z\in\R$. Take $N$ large enough such that
$-g(0)b_N^2\le z$ (this is possible as $b_N^2\to+ \infty$).
Now note that
\begin{align*}
 \widetilde\prob_N\left(\max_{\alpha\in V_N} Z_\alpha\le u_N(z)\right)& =\prod_{\alpha\in V_N} (1- \widetilde\prob_N( Z_\alpha>u_N(z)))\\
&\stackrel{\eqref{eq:Mills}}{\ge} \prod_{\alpha\in V_N} \left(1- \frac{\e^{- \frac{u_N(z)^2}{ 2g_N(\alpha)}}}{\sqrt{2\pi} u_N(z)} \sqrt{g_N(\alpha)}\right)\ge \left(1- \frac{\e^{- \frac{u_N(z)^2}{ 2g(0)}}}{\sqrt{2\pi} u_N(z)} \sqrt{g(0)}\right)^{N}.
\end{align*}
The last term converges to $\exp(-\e^{-z})$ as $N\to+ \infty$. This shows that for any fixed $z\in \R$,
$$\liminf_{N\to+ \infty}\widetilde\prob_N\left(\max_{\alpha\in V_N} \psi_\alpha\le u_N(z)\right)\ge \exp(-\e^{-z}).$$
\end{proof}
%We need some preliminary Lemmas for the upper bound. We begin with
%\begin{lemma}\label{lemma:almost_g}
%For any $\delta>0$ and $\alpha,\beta\in V_N^\delta$ one has
%\begin{equation}
%g(\alpha,\beta)-C_d\left(\delta N^{1/d}\right)^{2-d}\le g_N(\alpha,\beta)\le g(\alpha,\beta).
%\end{equation}
%In particular we have, $g_N(\alpha)=g(0)\left(1+\O{N^{(2-d)/d}}\right)$ uniformly in $\alpha\in V_N^\delta$.
%\end{lemma}
%\begin{proof}
%It follows from \citet[Proposition 1.6]{ASS} that
%$$g_N(\alpha,\beta)= g(\alpha,\beta)- \sum_{\gamma\in \partial V_N}\mathbb P_\alpha\left(H_{\Z^d\setminus V_N}<\infty, S_{H_{\Z^d\setminus V_N}}=\gamma\right)g(\gamma,\beta).$$
%Note that $g_N(\alpha, \beta)\le g(\alpha,\beta)$. Take any $\alpha,\beta\in V_N^\delta$: using the bounds~\eqref{eq:bound_g},
%$$\sum_{\gamma\in \partial V_N}\mathbb P_\alpha\left(H_{\Z^d\setminus V_N}<\infty, S_{H_{\Z^d\setminus V_N}}=\gamma\right)g(\gamma,\beta)\le \sup_{\gamma\in \partial V_N} g(\gamma, \beta)\le C_d\sup_{\gamma\in \partial V_N}\|\gamma-\beta\|^{2-d}$$
%which gives that
%\begin{equation}
%g_N(\alpha, \beta)\ge g(\alpha,\beta) - C_d \left(\delta N^{1/d}\right)^{2-d}.
%\end{equation}
%Hence the proof follows.
%\end{proof}
In order to prove the upper bound of Theorem~\ref{thm:main2}, we shall need a Lemma which will allow us to derive the convergence of the maximum in $V_N$ from that of the maximum in $V_N^\delta$.
\begin{lemma}\label{lemma:cot}
Let $N\ge 1$, $F_N$ be a distribution function, and $m_N=(1-2\delta)^d N$. Let $a_N$ and $b_N$ be as in~\eqref{eq:cs}. If $\lim_{N\to+\infty}F_N(a_{m_N}z+ b_{m_N})=\exp(-\e^{-z})$, then $$\lim_{N\to+\infty}F_N(a_Nz+b_N)=\exp\left(-\e^{-z+ d\log(1-2\delta)}\right).$$
\end{lemma}
\begin{proof}
The proof follows from a convergence-of-types theorem (see \citet[Proposition 0.2]{Resnick}) if we can show that
\begin{equation}\label{eq: cot1}
\frac{a_{m_N}}{a_N}\to 1 \quad\text{ and }\quad \frac{b_{m_N}-b_N}{a_N}\to d\log(1-2\delta).
\end{equation}
It is easy to see that
$$\frac{a_{m_N}}{a_N}\sim \left(1+\frac{d\log(1-2\delta)}{\log N}\right)^{1/2}\to 1.$$

To show the second asymptotics note that
\begin{equation}\label{eq:cot2}
\sqrt{2g(0)\log m_N}-\sqrt{2g(0)\log N}=\left[\frac{d\log (1-2\delta)}{2\log N}+ \O{\frac{1}{(\log N)^2}}\right]\sqrt{2g(0)\log N}.
\end{equation}
Also observe that as $N\to+ \infty$ one gets
\begin{align*}
&\sqrt{g(0)}\left[ \frac{\log \log  (4\pi N)}{2\sqrt{2\log N}}-\frac{\log \log (4\pi m_N )}{2\sqrt{2\log m_N}}\right]\\
%&=\frac{\sqrt{g(0)}}{2\sqrt{2\log N}}\left[\log \log ( 4\pi N)-\log[d\log(1-2\delta)+\log (4\pi N)]\left(1+\frac{d\log(1-2\delta)}{\log N}\right)^{-1/2}\right]\\
&=\frac{\sqrt{g(0)}}{2\sqrt{2\log N}}\left[-\log\left(1+\frac{d\log(1-2\delta)}{\log N}\right)+\o{1}\right].
\end{align*}
So using the above equation and \eqref{eq:cot2} we get that
\begin{align*}
\frac{b_{m_N}-b_N}{a_N}&=\frac{b_{m_N}-b_N}{g(0)}\sqrt{2g(0)\log N}(1+\o{1})\to d\log(1-2\delta).
%&=\left[\frac{d\log (1-2\delta)}{2\log N}+ \O{\frac{1}{(\log N)^2}}\right]2\log N(1+\o{1})+\\
%&+\left[-\log\left(1+\frac{d\log(1-2\delta)}{\log N}\right)+\o{1}\right]
\end{align*}
\end{proof}
We have now the tools to finish with the upper bound.
\begin{proof}[Proof of Theorem~\ref{thm:main2}: upper bound]
First fix $z\in \R$ and $\delta>0$, set $m_N:=\left|V_N^{\delta}\right|=(1-2\delta)^d N$. For the upper bound, we again use
%Theorem~\ref{thm:AGG}, but this time on $V_N^{\delta}$. We first observe that for any $\delta>0$
%$$\widetilde\prob_N\left( \max_{\alpha\in V_N}\psi_\alpha\le u_N(z)\right)\le\widetilde\prob_N\left(\max_{\alpha\in V_N^\delta} \psi_\alpha\le u_N(z)\right).$$
Lemma~\ref{fact:MP} and the fact that, for $\alpha\in V_N^\delta$, one has the equality $\vr_\alpha= \psi_\alpha+ \mu_\alpha^{(N)}$ under the infinite volume measure $\prob$,  where $\mu_\alpha^{(N)}= \E\left[ \vr_\alpha| \mathcal F_{\partial V_N}\right]
$.  Hence if we fix $\varepsilon>0$, and condition on the event that
$$\left\{\max_{\alpha\in V_N^\delta} |\mu_\alpha^{(N)}|\le \varepsilon a_{m_N}\right\}$$ (where $a_{m_N}$ is defined according to \eqref{eq:cs}) we have
\begin{align}
\widetilde \prob_N\left(\max_{\alpha\in V_N} \psi_\alpha\le u_{m_N}(z)\right)&\le \prob\left( \max_{\alpha\in V_N^\delta} \vr_\alpha\le u_{m_N}(z+\eps)\right)+ \prob\left( \max_{\alpha\in V_N^\delta} |\mu_\alpha|>\varepsilon a_{m_N}\right).\label{eq:upperboundfinite}
\end{align}
First we show that the second term goes to zero. Observe that $\mu_\alpha$ is a centered Gaussian with variance 
\[
\max_{\beta\in V_N^\delta}\var{\mu_\beta}\le \sup_{\beta\in V_N^\delta,\,\gamma\in \partial V_N}g(\beta,\,\gamma)=\O{N^{(2-d)/d}}.
\]
Let $(\Phi_\alpha)_{\alpha\in V_N^\delta}$ be a collection of i.i.d. Gaussians with mean zero and $\E\left[\Phi_\alpha^2\right]=\E\left[\mu_\alpha^2\right]$ for all $\alpha$. By Slepian and \citet[Prop. 1.~1.~3]{Tal03} we have
\begin{align*}
\prob\left( \max_{\alpha\in V_N^\delta} |\mu_\alpha|>\varepsilon a_{m_N}\right)&\le \frac{2\,\E\left[ \max_{\alpha\in V_N^\delta} \Phi_\alpha\right]}{a_{m_N}\varepsilon}+2\,\prob\left(\max_{\alpha\in V_N^\delta} \Phi_\alpha\le 0\right)\\
&\le\frac{2\sqrt{\max_{\beta\in V_N^\delta}\var{\mu_\beta} \log\left|V_N^\delta\right|}}{a_{m_N}\varepsilon}+\o{1}.
\end{align*}
Since $a_N$ grows like $\left(\sqrt{2\log N}\right)^{-1}$ as $N\to + \infty$,  we can conclude that, for every $\varepsilon>0$, \[\lim_{N\to+\infty}\prob\left( \max_{\alpha\in V_N^\delta} |\mu_\alpha|>\varepsilon a_{m_N}\right)= 0.\]
Using Theorem~\ref{thm:main} we have that
$$\lim_{N\to+\infty}\prob\left( \max_{\alpha\in V_N^\delta} \vr_\alpha\le u_{m_N}(z+\eps)\right)=\exp\left(-\e^{-(z+\eps)}\right),$$
and, being the limit continuous, we let $\varepsilon \to 0$ obtaining from~\eqref{eq:upperboundfinite}
\begin{equation}\label{eq:ineq}
\limsup_{N\to+ \infty}\widetilde \prob_N\left(\max_{\alpha\in V_N^\delta}\psi_\alpha\le u_{m_N}(z)\right)\le \exp(-\e^{-z}).
\end{equation}

Now using an easy comparison with independent random variables just as in the proof of the lower bound of Theorem~\ref{thm:main2} above it follows that~\eqref{eq:ineq} is in fact an equality.
%We claim that
%\begin{claim}\label{claim:conv:upper}
%For any fixed $z\in \R$ and $\delta>0$, set $m_N:=\left|V_N^{\delta}\right|=(1-2\delta)^d N$. Then
%$$\lim_{N\to+ \infty}\widetilde \prob_N\left(\max_{\alpha\in V_N^\delta}\psi_\alpha\le u_{m_N}(z)\right)=\exp(-\e^{-z}).$$
%\end{claim}
By Lemma~\ref{lemma:cot} one can conclude that $$\widetilde \prob_N\left(\max_{\alpha\in V_N^\delta} \psi_\alpha\le u_N(z)\right)=\exp\left(-\e^{-z+d\log(1-2\delta)}\right)$$
and thus letting $\delta\to 0$, the upper bound  follows.
%$$\limsup_{N\to+\infty}\widetilde \prob_N\left(\max_{\alpha\in V_N} \psi_\alpha\le u_N(z)\right)\le \exp\left(-\e^{-z+d\log(1-2\delta)}\right).$$

\end{proof}
\bibliographystyle{abbrvnat}

\bibliography{literatur}

\end{document}